\documentclass[a4paper,12pt]{article}
\usepackage{enumerate}
\usepackage{blkarray}
\usepackage{amsmath,amssymb,amsthm, mathrsfs}
\usepackage{latexsym,graphicx,kotex}
\usepackage{graphics}
\usepackage{hyperref}
\usepackage{mleftright}%vertical,horizontal line 
\usepackage{booktabs}
\usepackage[noadjust]{cite}
\usepackage[bf, small]{titlesec}
\usepackage{authblk} %%% new line between authors
\usepackage{soul}
\usepackage[displaymath,mathlines]{lineno}
\usepackage{cite}
\numberwithin{equation}{section}
\theoremstyle{plain}
\newtheorem{Thm}{Theorem}[section]
\newtheorem{Lem}[Thm]{Lemma}
\newtheorem{Prop}[Thm]{Proposition}
\newtheorem{Cor}[Thm]{Corollary}

\theoremstyle{definition}

\usepackage{standalone}
\usepackage{tikz}
\usetikzlibrary{arrows,positioning,matrix,fit,backgrounds,shapes,shapes.geometric, calc, intersections}

\usetikzlibrary{shapes.callouts,decorations.pathmorphing} %%% 팝업용
\usepackage{calc,stackrel} % calculate angles to evenly space vertices in circular arrangements.
\usepackage{geometry}
\usetikzlibrary{matrix,positioning,decorations.pathreplacing,decorations.markings}
\tikzstyle{vertex}=[circle, draw, inner sep=0pt, minimum size=6pt] % style

\usetikzlibrary{arrows,matrix} %%% Hesse Diagram

\title{On $(1,2)$-step competition graphs of multipartite tournaments II}

\author[]{Myungho Choi
\thanks{Corresponding author\\ 
E-mail addresses: nums8080@snu.ac.kr(M.Choi), srkim@snu.ac.kr(S.-R.Kim)}}
\author[]{Suh-Ryung Kim
}

\affil[]{Department of Mathematics Education,
Seoul National University, Seoul 08826, Republic of Korea}

 \newcounter{statement}
\newcommand{\statement}[2]{%
 \begin{equation}\refstepcounter{statement}\tag{S\thestatement}\label{#1}% 
  \parbox{\dimexpr\linewidth-4em}{#2}%
 \end{equation}%
}

\begin{document}
\maketitle
\begin{abstract}
A multipartite tournament is an orientation of a complete $k$-partite graph for some positive integer $k\geq 3$. 
We say that a multipartite tournament $D$ is tight if every partite set forms a clique in the $(1,2)$-step competition graph, denoted by $C_{1,2}(D)$, of $D$.
In the previous paper titled ``On $(1,2)$-step competition graphs of multipartite tournaments" \cite{choi202412step}
we completely characterize $C_{1,2}(D)$ for a tight multipartite tournament $D$.
%논문제목 적당성, arxiv 주소로 올리면 형식을 벗어남.
As an extension, in this paper, we study $(1,2)$-step competition graphs of multipartite tournaments that are not tight, which will be called loose.
For a loose multipartite tournament $D$, various meaningful results are obtained in terms of $C_{1,2}(D)$ being interval and $C_{1,2}(D)$ being connected.
\end{abstract}

% 05C20: Directed graphs (digraphs), tournaments
% 05C75: Structural characterization of families of graphs

\section{Introduction}
In this paper, all the graphs and digraphs are assumed to be finite and simple. (For all undefined
graph theory terminologies, see \cite{bondy}.)
Given a digraph $D$ and a vertex $v$ of $D$,
we define $N^+(v)=\{u \in V(D) \mid (v,u) \in A(D)\}$, $N^-(v)=\{u \in V(D) \mid (u,v) \in A(D) \}$, $d^+(v)=|N^+(v)|$, and $d^-(v)=|N^-(v)|$.
%When no confusion is likely to occur, we omit $D$ in $N^+_D(v)$ and $N^-_D(v)$ to just write $N^+(v)$ and $N^-(v)$.
We use the expression $u \to v$ when $(u,v) \in A(D)$. 
%If no confusion is likely, we just write $u \to v$.
When representing negation, add a slash ($/$) to the symbol.

For vertices $x$ and $y$ in a digraph $D$, $d_D(x,y)$ denotes the number of arcs in a shortest directed path from $x$ to $y$ in $D$ if it exists.
For positive integers $i$ and $j$, the {\it $(i,j)$-step competition graph } of a digraph $D$, denoted by $C_{i,j}(D)$, is a graph on $V(D)$ where $uv \in E(C_{i,j}(D))$ if and only if there exists a vertex $w$ distinct from $u$ and $v$ such that (a) $d_{D-v}(u,w) \leq i$ and $d_{D-u}(v,w) \leq j$ or (b) $d_{D-u}(v,w) \leq i$ and $d_{D-v}(u,w) \leq j$.
 If two vertices of a digraph $D$ are adjacent in $C_{1,2}(D)$, then we just say that they are adjacent in the rest of this paper.
 
The $(1,1)$-step competition graph of a digraph $D$ is the {\it competition graph} of $D$.
Given a digraph $D$, the competition graph of $D$, denoted by $C(D)$, is the graph having the vertex set $V(D)$ and the edge set $\{uv \mid u \to w, v\to w \text{ for some } w \in V(D) \}$.
Cohen~\cite{cohen1968interval} introduced the notion of competition graph while studying
predator–prey concepts in ecological food webs. Cohen’s empirical observation that real-world competition graphs are
usually interval graphs had led to a great deal of research on the structure of competition graphs and on the relation between
the structure of digraphs and their corresponding competition graphs. In the same vein, various variants of competition graph
have been introduced and studied, one of which is the notion of $(i, j)$-step competition introduced by Factor and Merz~\cite{factor20111}. For recent work on this topic, see
\cite{cohen1968interval,factor20111,kamibeppu2012sufficient,
kim2015generalization,kuhl2013transversals,li2012competition,
mckay2014competition,zhang20161,zhang2013note}.
Choi \emph{et al.}~\cite{choi20171} studied the $(1,2)$-step competition graph of an orientation of a complete bipartite graph.
In this paper, we study the $(1, 2)$-step competition graph of an orientation of a complete $k$-partite graph for some integer $k \geq 3$.

For a digraph $D$, we say that vertices $u$ and $v$ in $D$ $(1,2)$-{\it compete} provided there exists a vertex $w$ distinct from $u$, $v$ that satisfies one of the following:
\begin{itemize}
\item{} there exist an arc $(u,w)$ and a directed $(v,w)$-path of length $2$ not traversing $u$;
\item{} there exist a directed $(u,w)$-path of length $2$ not traversing $v$ and an arc $(v,w)$.
\end{itemize}
We call $w$ in the above definition a {\it $(1,2)$-step common out-neighbor }of $u$ and $v$.
 It is said that two vertices {\it compete} if they have a common out-neighbor. 
 If $u$ and $v$ compete or $(1,2)$-compete,
 then we say that $u$ and $v$
 $\{1,2\}$-{\it compete} in $D$. 
 We also say that a set is {\it $\{1,2\}$-competing} (resp.\ anti-$\{1,2\}$-competing) if any pair of vertices in the set $\{1,2\}$-competes (resp.\ if no two vertices in the set $\{1,2\}$-compete).

We note that a vertex set $S$ of a digraph $D$ is a $\{1,2\}$-competing set (resp.\ an anti-$\{1,2\}$-competing set) if and only if $S$ is a clique (resp.\ a stable set) in $C_{1,2}(D)$.
Here, a \emph{stable set} of a graph is a set of vertices no two of which are adjacent.
 
 Given vertex sets $X$ and $Y$ of a graph $G$,
 we use the symbol $X \sim Y$ if for each $x\in X$ and $y\in Y-\{x\}$, $x$ and $y$ are adjacent in $G$.
When $X$ or $Y$ is a singleton, the set brackets are not used. 
For example, we write $ u \sim Y$ if $X=\{u\}$. 

 If $u$ and $v$ are adjacent in $C_{1,2}(D)$ by competing (resp.\ $(1,2)$-competing), then it is simply represented as $u \sim_1 v$ (resp.\ $u \sim_{1,2} v$).
 We note that $u \sim v$ if and only if $u\sim_1 v$ or $u\sim_{1,2}v$.

We call an orientation of a complete $k$-partite graph for some positive integer $k$ a \emph{$k$-partite tournament}. 
In this paper, a \emph{multipartite tournament} refers to a $k$-partite tournament with $k \ge 3$.
A tournament of order $n$ may be regarded as an $n$-partite tournament for some positive integer $n$.
 For simplicity, we call a multipartite tournament with a non-$\{1,2\}$-competing partite set a {\it loose} multipartite tournament.
 
 In this paper, we study $(1, 2)$-step competition graphs of loose multipartite tournaments as an extension of the paper titled ``On $(1,2)$-step competition graphs of multipartite tournaments I" \cite{choi202412step}, which completely characterizes $C_{1,2}(D)$ when $D$ is a non-loose multipartite tournament. For a loose multipartite tournament $D$, various meaningful results are obtained in terms of $C_{1,2}(D)$ being interval and $C_{1,2}(D)$ being connected.
 Specifically, using structure features of $C_{1,2}(D)$ (Theorem~\ref{thm:structure}), we drive the following useful results: (a) each component of $C_{1,2}(D)$ has diameter of at most two; (b) $D$ has no sinks if and only if $C_{1,2}(D)$ is connected; (c) $D$ has at most two non-$\{1,2\}$-competing partite sets (Theorem~\ref{thm:Type-AorB-distance}).
 %specifically로 문단은 안 나눔.
We verify that if each partite set of $D$ has size at most two, then $C_{1,2}(D)$ is an interval graph (Theorems~\ref{thm:partite_size_interval}).
Finally, we prove that $C_{1,2}(D)$ is an interval graph if and only if $C_{1,2}(D)$ has no induced cycle of length four when $D$ has exactly two non-$\{1,2\}$-competing partite sets (Theorem~\ref{thm:C_4_equivalent}).

\section{Preliminaries}\label{property}
Let $D$ be a multipartite tournament.
We call a vertex of outdegree $0$ in a digraph $D$ a {\em sink} of $D$.
It is obvious that
each non-sink vertex has at least one out-neighbor.
For a non-sink vertex $u$ and a vertex $v$,
$u \stackrel{*}{\to} v$ means that $v$ is the only out-neighbor of $u$.

\begin{Prop}[\cite{choi202412step}]\label{prop:same-partite}
Let $D$ be a multipartite tournament and $u$ and $v$ be two non-sink vertices belonging to the same partite set in $D$. 
Then the following are true:
\begin{itemize}
\item[(1)] $u$ and $v$ do not compete if and only if $N^+(u) \cap N^+(v) = \emptyset$;
\item[(2)] $u$ and $v$ do not $(1,2)$-compete if and only if $N^+(u) \cup N^+(v) \subseteq X$ for some partite set $X$ of $D$.
\item[(3)] $u$ and $v$ are not adjacent in $C_{1,2}(D)$ if and only if $N^+(u) \cap N^+(v) = \emptyset$ and $N^+(u) \cup N^+(v) \subseteq X$ for some partite set $X$ of $D$.
\end{itemize}
\end{Prop}

\begin{Prop}[\cite{choi202412step}] \label{Prop:different-partite-adjacent}
Let $D$ be a multipartite tournament and $u$ and $v$ be two non-sink vertices with $(u,v)\in A(D)$ belonging to the distinct partite sets in $D$.
Then the following are equivalent:
\begin{itemize}
\item[(i)] $u$ and $v$ are not adjacent in $C_{1,2}(D)$;
\item[(ii)]
either
$u \stackrel{*}{\to} v$, or $N^+(u) \cap N^+(v) = \emptyset$ and there exists a partite set $X$ such that $N^+(v) \subseteq X$ and $N^+(u) \subseteq X \cup \{v\}$.
\end{itemize}
\end{Prop}

\begin{Cor}[\cite{choi202412step}]\label{Cor:cor-adajcent}
Let $D$ be a multipartite tournament and $u$, $v$ be two non-sink vertices in $D$.
Then $u$ and $v$ are adjacent in $C_{1,2}(D)$ if and only if the following happen:
\begin{enumerate}[{(1)}]
\item $v$ is not the only out-neighbor of $u$;
\item $u$ is not the only out-neighbor of $v$;
\item for any partite set $X$ with $N^+(v) \subseteq X$ (resp.\ $N^+(u) \subseteq X$),
 $N^+(u)\not\subseteq X \cup \{v\}$ (resp.\ 
 $N^+(v)\not\subseteq X \cup \{u\}$).
\end{enumerate}
\end{Cor}

\section{Structure of $C_{1,2}(D)$ for a loose multipartite tournament $D$}
\label{one-partite-set}
In this section, we characterize $(1,2)$-step competition graphs of multipartite tournaments with a non-$\{1,2\}$-competing partite set.
To this end, we need the following proposition.

\begin{Prop} \label{Prop:sub-digraph}
Let $D$ be a loose multipartite tournament with an anti-$\{1,2\}$-competing set $S$ of size at least two in a non-$\{1,2\}$-competing partite set $X$.
%If $S$ is contained in a partite set $X$ of $D$,
Then the following parts are valid:
  \begin{enumerate}[{(1)}]
   \item if $S$ is not precisely composed of one non-sink vertex and sink vertices, then $\bigcup_{u \in S} N^+(u)$ is contained in a partite set of $D$; 
    \item   if $S$ has at least three vertices, then any pair of vertices in $V(D)- X$ has a common out-neighbor in $S$.
    \end{enumerate}
\end{Prop}
\begin{proof}
%Suppose that $S$ is contained in a partite set $X$ of $D$.
%Let $X$ be the non-$\{1,2\}$-competing partite set containing $S$.
To show part (1),
suppose that $S$ is not precisely composed of one non-sink vertex and sink vertices.
By the hypothesis that $S$ is an anti-$\{1,2\}$-competing set, no two vertices in $S$ are adjacent.
If $S$ contains no sink,
then there exists a partite set $X'$ of $D$ such that $\bigcup_{u \in S}N^+(u) \subseteq X'$ by Proposition~\ref{prop:same-partite}(3).
Suppose that $S$ contains a sink.
Then there exist at least two non-sink vertices $x$ and $y$ in $S$ and so, by Proposition~\ref{prop:same-partite}(3), $N^+(x) \cup N^+(y) $ is contained in a partite set $X'$ of $D$.
Since $x$ and $y$ were arbitrarily chosen from $S$,
each non-sink vertex in $S$ has an out-neighborhood in $X'$.
Thus we have shown that part (1) is true.

To show the part (2),
we assume $|S| \geq 3$.
Since no two vertices in $S$ are adjacent in $C_{1,2}(D)$,
each vertex in $V(D) - X$ has at least $|S|-1$ out-neighbors in $S$.
Since $|S| \geq 3$, $2(|S|-1) > |S|$ and so, by the Pigeonhole principle, there exists a common out-neighbor in $S$ of  each pair of vertices in $V(D)- X$.
Therefore the part (2) is true.
\end{proof}

From now on, when we mention a $k$-partite tournament $D$ with a non-$\{1,2\}$-competing partite set $X_1$ for some $k \geq 3$, we assume that $X_2,\ldots,X_k$ are the remaining partite sets of $D$.
If $D$ has a sink in a partite set distinct from $X_1$,
then the sink is a common out-neighbor of each vertex in $X_1$ and so $X_1$ is a clique, which is a contradiction.
Thus 
%\begin{itemize}
%\item 
%\end{itemize}
\statement{sta:thm:sink-hypothesis}{if $D$ has a sink $u$, then $u$ is contained in $X_1$.}
\begin{Prop} \label{prop:distinct-partitesets}
Let $D$ be a loose $k$-partite tournament $D$ and $X_1$ be a non-$\{1,2\}$-competing partite set for some $k \geq 3$.
If a vertex in $X_1$ has out-neighbors in distinct partite sets, then the vertex is adjacent to all the non-sink vertices in $C_{1,2}(D)$.
\end{Prop}
\begin{proof}
{\it Case 1}. $D$ has a sink $u$.
Then, by \eqref{sta:thm:sink-hypothesis}, \begin{equation}\label{eq:thm:out-degree-0-charac}
N^-(u)=V(D) - X_1.
\end{equation}
 Suppose that a vertex $v$ in $X_1$ has two out-neighbors $w$ and $x$ belonging to distinct partite sets. Then we take a non-sink vertex $y$. If $y \in V(D)- X_1$, then $u$ is an out-neighbor of $y$ and so, by Corollary~\ref{Cor:cor-adajcent}, $v$ and $y$ are adjacent. If $y \in X_1$, by Proposition~\ref{prop:same-partite}(3),  $v \sim y$ since $N^+(v) \not \subseteq X$ for any partite set $X$ of $D$.

{\it Case 2}. $D$ has no sinks.
Since $X_1$ is a non-$\{1,2\}$-competing partite set,
 we may take an anti-$\{1,2\}$-competing set $S$ included in $X_1$ with $|S|\geq 2$.
Suppose that a vertex $u$ in $X_1$ has two out-neighbors belonging to distinct partite sets.
 Then $u \notin S$ by Proposition~\ref{Prop:sub-digraph}(1).
Moreover, $u \sim X_1$ by Proposition~\ref{prop:same-partite}(2).
Since $|S| \geq 2$, each vertex in $V(D) - X_1$ has at least one out-neighbor in $S$.
Therefore $u \sim (V(D)-X_1)$ by  Proposition~\ref{Prop:different-partite-adjacent}.
\end{proof}
According to the above proposition, in order to characterize $C_{1,2}(D)$ for a multipartite tournament $D$ with a non-$\{1,2\}$-competing partite set $X_1$, we only need to take a look at a vertex in $X_1$ with all of its out-neighbors in the same partite set.

By Proposition~\ref{Prop:sub-digraph}(1), we may assume the following:
\statement{sta:thm:competing_set_hypothesis}{if a $k$-partite tournament has an anti-$\{1,2\}$-competing set $S \subseteq X_1$ that contains at least two non-sink vertices, then 
each non-sink vertex in $S$ has an out-neighborhood in $X_2$.}
%{\red 
%(S2) 사용된 곳 찾아서 바꾸기
%}

\begin{Thm} \label{thm:new-neighbor-stable-size2}
   Let $D$ be a loose $k$-partite tournament and $X_1$ be a non-$\{1,2\}$-competing partite set in $D$.
    Then the following parts are valid:
    \begin{enumerate}[{(1)}]
    \item
        for $2\leq i \leq k$,
        $\{v:\emptyset \neq N^+(v)\subseteq X_i\}\subseteq X_1$;
%        each vertex $v$ with $N^+(v) \subseteq X_i$ belongs to $X_1$;
\item $\bigcup_{i=3}^k X_i$ forms a clique in $C_{1,2}(D)$ and
$\bigcup_{i=3}^k X_i \sim X_2$;
\item $\{v : \emptyset \neq N^+(v) \subseteq X_i\} \sim \big(\{v : \emptyset \neq N^+(v) \subseteq X_j\} \cup X_j\big)$ for distinct integers $2\leq i, j \leq k$
\item given a non-sink vertex $x$ with $N^+(x)\subseteq X_i$ and a vertex  $y$ in $X_i$ for some $i \in \{2,\ldots,k\}$,
$x\not\sim y$ if and only if $x \stackrel{*}{\to} y $ or
    $y \stackrel{*}{\to} x $;

  \item if $D$ has a sink or an anti-$\{1,2\}$-competing set of size at least three in $X_1$, then $V(D)-X_1$ forms a clique in $C_{1,2}(D)$.
    \end{enumerate}
    \end{Thm}

    \begin{proof}

     We first show the parts (1) and (2).
Let $S$ be an anti-$\{1,2\}$-competing set with size at least two in $X_1$. Then
each vertex in $\bigcup_{i=2}^k X_i$ has an out-neighbor in $S \subseteq X_1$.
Therefore $\{v:\emptyset \neq N^+(v) \subseteq X_i\}\cap \bigcup_{i=2}^k X_i  =\emptyset$ for any $2\leq i\leq k$ and so
the part (1) is true.
%no 해서, 뒤에 any로

For simplicity,
let $F_i=\{v: \emptyset \neq N^+(v) \subseteq X_i\}$ for each $2\leq i \leq k$.
If $S$ contains a sink $u$, then $u$ is an out-neighbor of each vertex in $\bigcup_{i=2}^k X_i$ by \eqref{sta:thm:sink-hypothesis} and so $\bigcup_{i=2}^k X_i$ forms a clique in $C_{1,2}(D)$.
Suppose that $S$ does not contain a sink.
Then $S\subseteq F_2$ by \eqref{sta:thm:competing_set_hypothesis}.
Therefore
the out-neighborhood of each vertex in $S$ is included in $X_2$.
Thus each vertex in $S$ is an out-neighbor of each vertex in $\bigcup_{i=3}^k X_i$.
Hence $\bigcup_{i=3}^k X_i$ forms a clique in $C_{1,2}(D)$.
In addition, since each vertex in $X_2$ has an out-neighbor in $S \subseteq X_1$, $\left(\bigcup_{i=3}^k X_i\right) \sim X_2$ and so the part (2) is true.

To show the part (3), take two vertices $v\in F_i$ and $x \in  X_j \cup F_j $ for distinct $i$ and $j$ in $\{2,\ldots,k\}$.
    If $x \in F_j$, then $\{v,x\} \subseteq X_1$ by the part (1) and so, by Proposition~\ref{prop:same-partite}(2), $v \sim_{1,2} x$ and so $v \sim x$.
    Suppose $ x\in X_j$.
        Then $x$ is not an out-neighbor of $v$.
      Therefore $v$ is an out-neighbor of $x$.
      Moreover, since $v\in F_i$,
      $v$ has an out-neighbor distinct from $x$.
      Thus, by Corollary~\ref{Cor:cor-adajcent}, it is suffices to show that $x$ has an out-neighbor distinct from $v$.
      If $S$ has a sink $u$, then $u$ is an out-neighbor of $x$ distinct from $v$ and we are done.
      Therefore we assume that $S$ has no sinks.
      %we suppose을  we assume으로 변경함.

        Suppose $j \neq 2$. Then, since $S\subseteq F_2$ by \eqref{sta:thm:competing_set_hypothesis}, each vertex in $S$ cannot have $x\in X_j$ as an out-neighbor.
         Therefore $S \subseteq N^+(x)$. 
         Thus $x$ has at least two out-neighbors in $S \subseteq X_1$ since $|S| \geq 2$.
        Hence $x$ has an out-neighbor in $X_1$ distinct from $v$.
        Now we suppose $j =2$. Then $i\neq2$.
        Since $|S|\geq 2$,
       $x$ has at least one out-neighbor $x'$ in $S$.
       Since $S \subseteq F_2$ by the assumption,
       $x'$ is distinct from $v$.
        In each case, $x$ has an out-neighbor in $X_1$ distinct from $v$.
        Therefore
        $v$ and $x$ are adjacent by Corollary~\ref{Cor:cor-adajcent}.
   Since $v$ and $x$ were arbitrarily chosen from $F_i$ and $X_j$, respectively, the part (3) is true.

The ``if" part of the part (4) is true by Proposition~\ref{Prop:different-partite-adjacent}.
To show the ``only if" part of the part (4),
suppose, to the contrary, that there exist a vertex $x$ in $F_i$ and a vertex $y$ in $X_i$ for some $i \in \{2,\ldots,k\}$ such that they are not adjacent, $N^+(x) \neq \{y\}$, and $N^+(y) \neq \{x\}$.
Then, since $x$ and $y$ are non-sink vertices,
$N^+(x) - \{y\} \neq \emptyset$ and
$N^+(y) - \{x\} \neq \emptyset$.
Since $y \in X_i$,
$(N^+(y) - \{x\}) \cap X_j \neq \emptyset$ for some $j$ distinct from $i$.
Thus $N^+(y) \not \subseteq X_i \cup \{x\}$.
Therefore $x\sim y$ by Corollary~\ref{Cor:cor-adajcent}, a contradiction.

The part (5) is an immediate consequence of \eqref{eq:thm:out-degree-0-charac} and
Proposition~\ref{Prop:sub-digraph}(2).\end{proof}

Proposition~\ref{prop:distinct-partitesets} and Theorem~\ref{thm:new-neighbor-stable-size2} may be summarized as follows.

\begin{Thm}\label{thm:structure}
Let $D$ be a loose multipartite tournament, $X_1$ be a non-$\{1,2\}$-competing partite set, and $U$ be the set of sinks in $D$ ($U$ is possibly vacuous).
Then the adjacency matrix of $C_{1,2}(D)$ is in the form given in Figure~\ref{fig:structure}.	
\end{Thm}

\begin{figure}
\[
\begin{blockarray}{ccccccc|cccc}
& U &  F_2 & F_3 & \cdots  & F_k & X^*_1 & X_2 & X_3 & \cdots & X_k \\
\begin{block}{c(cccccc|cccc)} %| vertical line
	U& {O} & {O} & O & \cdots & {O} & {O}&  {O} & {O} & \cdots & {O} \\
    F_2& {O} & {\star} & J& {\cdots} & {J} & {J} & {\star}  & J & {\cdots} & {J} \\
    F_3& {O} & {J} & \star & {\cdots} & {J} & {J} & {J} & \star & {\cdots} & {J} \\
  {\vdots} &    {\vdots} & {\vdots}  & \vdots & {\ddots}  & {\vdots}  & {\vdots}  & {\vdots} & \vdots &{\ddots} & {\vdots} \\
   F_k&  {O} & {J} & {J} & {\cdots} & {\star} & {J} & {J} & {J} & {\cdots} & {\star} \\
 X^*_1 &    {O} & {J} & {J} & {\cdots} & {J} & {J-I} & {J} & {J} & {\cdots} & {J} \\ 
 \cmidrule{1-11}%horizontal line
  X_2&   {O} & {\star} & {J} & {\cdots} & {J} & {J} & {M} & {J} & {\cdots} & {J} \\
  X_3& {O} & {J} & \star & {\cdots} & {J} & {J} & {J} & J-I & {\cdots} & {J} \\
 {\vdots} &    {\vdots} & {\vdots}  & \vdots & {\ddots}  & {\vdots}  & {\vdots}  & {\vdots} & \vdots &{\ddots} & {\vdots} \\
  X_k &   {O} & {J}& {J} & {\cdots} & {\star} & {J} & {J} &J  & {\cdots} & {J-I} \\
\end{block}
\end{blockarray}
\]
\caption
{The adjacency matrix $A$ of $C_{1,2}(D)$ for a multipartite tournament $D$ with a non-$\{1,2\}$-competing partite set $X_1$ where $U$ is the set of sinks in $D$ ($U$ is possibly vacuous); $F_i$, $O$, $J$, and $I$ stand for $\{v: \emptyset \neq N^+(v) \subseteq X_i\}$, a zero matrix, a matrix of all $1$'s, and an identity matrix, respectively; $X^*_1=X_1- \left(\bigcup_{i=2}^kF_i \cup U  \right)$; $M$ is undetermined, yet, if $D$ has a sink or an anti-$\{1,2\}$-competing set of size at least three in a partite set, then $M=J-I$; Blocks marked with $\star$ depend on $D$.}
\label{fig:structure}
\end{figure}

A set $S$ of vertices in a graph $G$ is called a {\it dominating set} if every vertex $v \in V$ is either an element of $S$ or is adjacent to an element of $S$.
The {\it domination number} $\gamma(G)$ of a graph $G$ equals the minimum cardinality of a dominating set in $G$.

Given a digraph $D$,
each sink in $D$ is isolated in $C_{1,2}(D)$ and so $m \leq \gamma(C_{1,2}(D))$ where $m$ is the number of sinks in $D$. The following theorem gives upper bound for the domination number of $C_{1,2}(D)$ when $D$ is a loose multipartite tournament.

\begin{Thm} \label{thm:Type-AorB-distance}
Let $D$ be a loose multipartite tournament.
Then the following are true:
 \begin{enumerate}[{(1)}]
 \item each component of $C_{1,2}(D)$ has diameter of at most two;
 \item $D$ has no sinks if and only if $C_{1,2}(D)$ is connected;

 \item each stable set of $C_{1,2}(D)$ is contained in at most two partite sets of $D$;

\item $D$ has at most two non-$\{1,2\}$-competing partite sets;

\item if $D$ has $m$ sinks, then $\gamma(C_{1,2}(D)) \in \{m+1,m+2\}$ unless $C_{1,2}(D)$ is an empty graph.
 \end{enumerate}
\end{Thm}
\begin{proof}
By Theorem~\ref{thm:structure}, the adjacency matrix $A$ of $C_{1,2}(D)$  is in the form of the matrix given in Figure~\ref{fig:structure}.
Therefore it is easy to see from $A$ that any pair of vertices of each component is at distance at most two and so the part (1) is true.
Moreover, $C_{1,2}(D)$ is connected if and only if $U = \emptyset$, which can be seen from $A$.
%by the matrix => can be seen from 정정
Recall that $U$ is the set of sinks in $D$.
Therefore $D$ has no sinks if only if $U= \emptyset$.
Thus the part (2) is true.
Now we prove the part (3).
%is intersecting -> intersects 주어가 생물이 아니니
Since $D$ has a non-$\{1,2\}$-competing partite set,
 a stable set $S$ of $C_{1,2}(D)$ intersects with at most one partite set among $X_2,\ldots, X_k$ by Theorem~\ref{thm:new-neighbor-stable-size2}(2).
Therefore $S$ intersects with at most two partite sets of $D$.
By the structure of $A$, $X_1$ and $X_2$ are the only possible non-$\{1,2\}$-competing partite sets of $D$. 
Thus the part (4) is true.

To show part (5), suppose that $D$ has $m$ sinks and $C_{1,2}(D)$ is not an empty graph.
Then $\gamma(C_{1,2}(D)) \geq m+1$. 
Since $D$ is a loose multipartite tournament, by Theorem~\ref{thm:structure}, the adjacency matrix $A$ of $C_{1,2}(D)$ is in the form given in Figure~\ref{fig:structure}.
We take a vertex $u$ in $X_2$ and a vertex $v$ in $X_3$.
By the structure of $A$, any vertex except sinks is adjacent to $u$ or $v$ in $C_{1,2}(D)$.
Therefore $\gamma(C_{1,2}(D))\leq m+2$.
\end{proof}

Motivated by Theorem~\ref{thm:Type-AorB-distance}(3),
we may ask the question
``Given a multipartite tournament $D$, what is the size of a biggest set among the anti-$\{1,2\}$-competing sets that are not included in any partite set of $D$?".
If a multipartite tournament $D$ has no sinks, then a biggest set among such anti-$\{1,2\}$-competing sets of $D$ has the size at most four, which is told by Theorem~\ref{different-partite-stable-set-at most four}.

To justify Theorem~\ref{different-partite-stable-set-at most four}, we need the following lemma.

\begin{Lem}[\cite{choi202412step}]\label{lem:cycle-adjacent}
Let $D$ be a digraph having a directed cycle $C$ of order $l$ for some $l \in \{3,4\}$ and $X$ be a subset of $V(D)$.
Suppose that each vertex $u$ in $X - V(C)$ has two out-neighbors $u_1$ and $u_2$ on $C$ such that both $(u_1,u_2)$-section and $(u_2,u_1)$-section of $C$ have length at most $2$.
Then $(X-V(C)) \sim X$. 
\end{Lem}

\begin{Lem} \label{lem:partiteset-size-three}
Let $D$ be a multipartite tournament.
Suppose that $D$ has no sinks and there exists an anti-$\{1,2\}$-competing set of size three contained in a partite set $X$. 
Then every anti-$\{1,2\}$-competing set of size at least three is contained in $X$.
\end{Lem}

\begin{proof}
Let $S$ be an anti-$\{1,2\}$-competing set of size three contained in $X$. 
Take an anti-$\{1,2\}$-competing set $S'$ of size at least three in $D$.
Then, by Proposition~\ref{Prop:sub-digraph}(2), each pair of vertices in $S'\cap (V(D)-X)$ has a common out-neighbor in $S$ and so $|S'\cap (V(D)-X)|\leq 1$.
If $S'\cap (V(D)-X)=\emptyset$, then $S' \subseteq X$ and so we are done.
Suppose, to the contrary, that $|S'\cap (V(D)-X)|=1$.
Then $|S'\cap X| \geq 2$.
Take vertices $v_1$ and $v_2$ in $S'\cap X$ and $v_3$ in $S'\cap (V(D)-X)$.
Let $X'$ be the partite set of $D$ containing $v_3$.
If $v_1 \to v_3$ and $v_2 \to v_3$, then $v_1$ and $v_2$ compete, which contradicts the fact that $S'$ is an anti-$\{1,2\}$-competing set.
Thus $v_1 \not \to v_3$ or $v_2 \not \to v_3$.
Without loss of generality, we may assume $v_1 \not \to v_3$. Then $v_3 \to v_1$.
Since $D$ has no sinks, $v_1$ has an out-neighbor $v_1'$  belonging to $V(D)-X$.
Moreover, since $S$ is an anti-$\{1,2\}$-competing set of size three contained in $X$,
at least two vertices of $S$ are out-neighbors of $v_3$.
Thus $v_3$ has an out-neighbor $v_3'$ in $S$ distinct from $v_1$.
If $v_3' \to v_1'$ (resp.\ $v_1' \to v_3'$), then $v_1'$ (resp.\ $v_3'$) is a $(1,2)$-step common out-neighbor of $v_3$ and $v_1$, and so $v_1 \sim v_3$, a contradiction.
\end{proof}

\begin{Thm} \label{different-partite-stable-set-at most four}
Let $D$ be a loose multipartite tournament of order $n$.
Suppose $D$ has no sinks. If $D$ has an anti-$\{1,2\}$-competing set $S$ which is not included in any partite set of $D$,
then $|S| \leq 4$.
Especially, if $|S| =4$, then the following are true:
\begin{enumerate}[{(1)}]
\item there exist two partite sets $X_1$ and $X_2$ of $D$ such that $|S \cap X_1|=|S\cap X_2|=2$;
\item $n\geq 5$ and  $C_{1,2}(D) \cong K_{n}-E(K_4)$.
\end{enumerate}
\end{Thm}

\begin{proof}
Suppose that $D$ has an anti-$\{1,2\}$-competing set $S$ which is not included in any partite set of $D$. 
Let $X_1,\ldots,X_k$ be the partite sets of $D$ and \[\Lambda=\{i \mid S \cap X_i \neq \emptyset\}.\]
Since there is no partite set of $D$ including $S$, $|\Lambda| \geq 2$.
Then $|\Lambda| = 2$ by Theorem~\ref{thm:Type-AorB-distance}(3).
Thus, by Lemma~\ref{lem:partiteset-size-three}, there is no partite set containing at least three vertices in $S$ and so $|S|\leq 4$. 
In addition, we have shown that if $|S|=4$, then there exist two partite sets $X_1$ and $X_2$ of $D$ such that $|S \cap X_1| = |S \cap X_2|=2$ and so (1) of the ``especially" part is true.

To show (2) of the ``especially" part, suppose $|S|=4$.
As we have shown (1) of the ``especially" part, there exist two partite sets $X_1$ and $X_2$ of $D$ such that $|S \cap X_1| = |S \cap X_2|=2$.
Then $D$ has a non-$\{1,2\}$-competing partite set.
In addition, since $k \geq 3$ by the definition of multipartite tournament, there exists the partite set $X_3$ with at least one vertex and so \[n\geq 5.\]
We note that $S \cap X_1:=\{u_1,u_2\}$ and $S \cap X_2:=\{u_3,u_4\}$ are stable sets of size two in  $C_{1,2}(D)$.
Therefore
neither $u_1$ nor $u_2$ is a common out-neighbor of $u_3$ and $u_4$ and vice versa.
Without loss of generality, we may assume \[\{(u_1,u_3), (u_2,u_4),(u_3,u_2),(u_4,u_1)\} \subset A(D).\]
$N^+(u_1) \cup N^+(u_2) \subseteq X_2$ and $N^+(u_3) \cup N^+(u_4) \subseteq X_1$ by Proposition~\ref{Prop:sub-digraph}(1).
Now, since $D$ has a non-$\{1,2\}$-competing partite set,
by Theorem~\ref{thm:new-neighbor-stable-size2}(4), \[
u_1\stackrel{*}{\to} u_3, \quad
u_2 \stackrel{*}{\to} u_4
 \quad
 u_3 \stackrel{*}{\to} u_2,
 \quad \text{and} \quad
 u_4 \stackrel{*}{\to} u_1.\]
Then the vertices in $S$ form a directed cycle $C:=u_1 \to u_3 \to u_2 \to u_4 \to u_1$ of order $4$ in $D$.
Take a vertex $x$ in $V(D) - V(C)$.
If $x \in X_1$ or $x \in X_2$  , then $\{u_3,u_4\} \subseteq N^+(x)$ or $\{u_1,u_2\} \subseteq N^+(x)$.
If $x \notin X_1 \cup X_2$, then $V(C) \subseteq N^+(x)$.
Therefore $x$ has two out-neighbors $y$ and $z$ on $C$ such that both $(y,z)$-section and $(z,y)$-section of $C$ have length $2$.
Since $x$ was arbitrarily taken from $V(D) - V(C)$,
we conclude that
each vertex in $V(D) - V(C)$ has two out-neighbors in $S$ satisfying the condition given in Lemma~\ref{lem:cycle-adjacent}.
Therefore $\left(V(D)-V(C)\right)\sim V(D)$.
Then, since $S=\{u_1,u_2,u_3,u_4\}$ is a stable set, $C_{1,2}(D) \cong K_n - E(K_4)$.
\end{proof}

\section{Interval $(1,2)$-step competition graphs}
A graph is an {\it interval graph} if we can assign each vertex $v$ a real interval $J(v) \subseteq \mathbb{R}$ such that there is an edge between two distinct vertices $v$ and $w$ if and only if $J(v) \cap J(w) \neq \emptyset$.
	Given a graph,
an induced cycle of length at least four is called a {\it hole}.
A {\it chordal graph} is a graph without holes. 
An {\it asteroidal triple} of a graph is a set of three vertices such that every pair of its vertices are joined by a path outside of the closed neighborhood of the third.
Lekkeikerker and Boland~\cite{lekkeikerker1962representation} proved that a graph is an interval graph if and only if it is chordal and contains no asteroidal triple.

\begin{Lem} \label{lem-asterodial}
Let $D$ be a loose multipartite tournament with a non-$\{1,2\}$-competing partite set.
An asteroidal triple of $C_{1,2}(D)$, if any, is contained in a partite set
$X$ that is the only non-$\{1,2\}$-competing partite set in $D$.
\end{Lem}
\begin{proof}
Suppose that $C_{1,2}(D)$ has an asteroidal triple $x$, $y$, and $z$.
Then $\{x,y,z\}$ forms a stable set of size three in $C_{1,2}(D)$.
By Theorem~\ref{thm:structure}, the adjacency matrix $A$ of $C_{1,2}(D)$ is in the form given in Figure~\ref{fig:structure}.
%{\red We use the symbols if they appear in Figure~\ref{fig:structure}.
%}
Then, by the structure of $A$, it is easy to check that $\bigcup_{i=3}^k X_i$ is a clique in $C_{1,2}(D)$.
In addition, $\bigcup_{i=3}^k X_i \sim X_2$ by Theorem~\ref{thm:new-neighbor-stable-size2}(2)
Therefore at least two vertices in $\{x,y,z\}$ are contained in a partite set $X_1$ or $X_2$ of $D$.
If at least two vertices in $\{x,y,z\}$ are contained in $X_2$,
then $X_2$ is non-$\{1,2\}$-competing.
Therefore $X_1$ and $X_2$'s roles can be exchanged in Theorem~\ref{thm:structure} and so we may assume $\{x,y\} \subset X_1$.
We note that $x$, $y$, and $z$ are not isolated in $C_{1,2}(D)$ and so they are not sinks in $D$.
Since $x$ and $y$ are not adjacent,
$\{x,y\} \subset F_j$ (recall that $F_j=\{v:\emptyset \neq N^+(v) \subseteq X_j\}$) for some $j \in \{2,\ldots,k\}$.
Then 
\begin{equation}\label{eq:lem-asterodial_1}
N^+(x)\cup N^+(y)\subseteq X_j .\end{equation}
Moreover, $z \in F_j \cup X_j$ by the structure of $A$.
Suppose, to the contrary, that $z \notin F_j$.
Then $z\not\in X_1$.
We note that $x$ and $y$ have no common out-neighbor, $z \not \sim x$, and $z\not \sim y$.
Then we may assume
$x \stackrel{*}{\to} z$ and $z \stackrel{*}{\to} y$ by \eqref{eq:lem-asterodial_1} and Theorem~\ref{thm:new-neighbor-stable-size2}(4).
By the way, since $\{x,y,z\}$ is an asteroidal triple, $y$ and $z$ are connected by a path $P$ avoiding the closed neighborhood of $x$.
Let $w$ be a vertex on $P$ which is adjacent to $z$.
Then $w \not \sim x$ and so, by the structure of $A$, $w \in F_j$ or $w \in X_j$.
Suppose $w \in X_j$.
Then, since $x \stackrel{*}{\to} z$, $w\to x$.
Thus, by \eqref{eq:lem-asterodial_1} and Proposition~\ref{Prop:different-partite-adjacent}, either $w \stackrel{*}{\to} x$ or $N^+(x)\cap N^+(w)=\emptyset$ and $N^+(w)\subseteq X_j\cup \{x\}$.
Since $w \in X_j$,
$w \stackrel{*}{\to} x$.
Then, since $z \stackrel{*}{\to} y$, $w$ and $z$ have no common out-neighbor.
Moreover, $N^+(w) \cup N^+(z) \subseteq X_1$ and so they are not adjacent in $C_{1,2}(D)$ by Proposition~\ref{prop:same-partite}(3), a contradiction.
Therefore $w \notin X_j$ and so $w \in F_j$.
Thus there is an arc between $w$ and $z$.
Then, since $z \stackrel{*}{\to} y$, $(w,z) \in A(D)$ and so $z$ is a common out-neighbor of $w$ and $x$.
Hence $w \sim x$, a contradiction.
Therefore $\{x,y,z\} \subset X_1$. 
Thus $V(D)- X_1$ forms a clique by Theorem~\ref{thm:new-neighbor-stable-size2}(5) and so
$X$ is the only non-$\{1,2\}$-competing partite set in $D$.
\end{proof}
Let $G$ be a graph. Two vertices $u$ and $v$ of $G$ are said to be {\it true twins} if they have the same closed neighborhood. We may introduce an analogous notion for a digraph. Let $D$ be a digraph. Two vertices $u$ and $v$ of $D$ are said to be {\it true twins} if they have the same open out-neighborhood and open in-neighborhood in $D$. 

\begin{Lem}[\cite{choi202412step}]\label{lem:true-twin}
If two non-sink vertices are true twins in a digraph $D$, then they are true twins in $C_{1,2}(D)$.
\end{Lem}

\begin{Lem}\label{lem:x_i-other-adjacent}
Let $D$ be a loose multipartite tournament and let $u$ and $v$ be two adjacent vertices in the same non-$\{1,2\}$-competing partite set.
If $u$ and $v$ are not true twins in $C_{1,2}(D)$ and $N^+(u) \cup N^+(v) \subseteq X_i$ for some $i\in \{1,\ldots,k\}$,
then each vertex in $X_i$ is adjacent to at least one of $u$ and $v$.
\end{Lem}
\begin{proof}
Suppose that $u$ and $v$ are not true twins in $C_{1,2}(D)$ and $N^+(u) \cup N^+(v) \subseteq X_i$ for some $i\in \{1,\ldots,k\}$.
To the contrary, suppose there is a vertex $w$ in $X_i$ not adjacent to both $u$ and $v$.
If $u \stackrel{*}{\to} w$ and $v \stackrel{*}{\to} w$, then
$u$ and $v$ are true twins in $D$ and so,  by Lemma~\ref{lem:true-twin}, they are true twins in $C_{1,2}(D)$, a contradiction.
Therefore
$u \not\stackrel{*}{\to} w$ or $v \not \stackrel{*}{\to} w$.
Without loss of generality,
we may assume $u \not \stackrel{*}{\to} w$.
Then $w\stackrel{*}{\to} u$ by Theorem~\ref{thm:new-neighbor-stable-size2}(4).
Therefore $w \not \stackrel{*}{\to} v$ and so, by the same theorem, $v \stackrel{*}{\to} w$.
Thus $u$ and $v$ has no common out-neighbor.
Hence $u$ and $v$ are not adjacent in $C_{1,2}(D)$ by Proposition~\ref{prop:same-partite}(3), a contradiction.
\end{proof}
\begin{Lem}
\label{lem:x_i-other-adjacent_2}
Let $D$ be a multipartite tournament.
If there are two adjacent vertices which are not true twins in a non-$\{1,2\}$-competing partite set $X$ of $D$,
then any vertex adjacent to none of the two vertices in $C_{1,2}(D)$ belongs to $X$.
\end{Lem}
\begin{proof}
Suppose that there exist two adjacent vertices $u,v$ which are not true twins in a non-$\{1,2\}$-competing partite set $X$ of $D$ and
a vertex $w$ not adjacent to any of them.
Then $u$ and $v$ are non-sink vertices.
In addition, since $X$ is non-$\{1,2\}$-competing, the adjacency matrix $A$ of $C_{1,2}(D)$ is in the form given in Figure~\ref{fig:structure} by Theorem~\ref{thm:structure} in which $X=X_1$ was assumed.
Suppose $w \notin X_1$. 
Then $w \in X_i$ for some $i \in \{2,\ldots,k\}$.
Thus, by the structure of $A$, 
$\{u, v \}\subseteq F_i$.
Hence $w$ is adjacent to one of $u$ and $v$ by Lemma~\ref{lem:x_i-other-adjacent}, which is a contradiction.
Therefore $w \in X_1$ and so the statement is true.
\end{proof}

\begin{Thm} \label{thm:hole-charc}
Let $D$ be a loose multipartite tournament.
If $C_{1,2}(D)$ has a hole $H$ of length at least five, 
then there exists a partite set $X$ satisfying the following:

\begin{enumerate} [{(1)}]
       \item $X$ is the only non-$\{1,2\}$-competing partite set;
       \item
      for every hole $L$ of length at least five, there exists a partite set $Y$ distinct from $X$ such that $|Y| \geq |V(L)| $ and
       $V(L) \subseteq \{v:\emptyset \neq N^+(v) \subseteq Y \} \subseteq X$.
       \end{enumerate}
       \end{Thm}
\begin{proof}
Since $D$ has a non-$\{1,2\}$-competing partite set,
the adjacency matrix $A$ of $C_{1,2}(D)$ is in the form given in Figure~\ref{fig:structure} by Theorem~\ref{thm:structure} in which $X_1$ was assumed to be a non-$\{1,2\}$-competing partite set of $D$.
Suppose that $C_{1,2}(D)$ has a hole $H=v_0v_1 \ldots v_{l-1}v_0$ of length $l\geq 5$.
Then $v_0,v_1,\ldots,v_{l-1}$ are not sinks in $D$.

We will prove part (1) by considering two cases.

%고칠 것 : the statement => the part
{\it Case 1}. $l \geq 6$.
Then there exists an asteroidal triple $\{x,y,z\}$ on $H$ and so, by Lemma~\ref{lem-asterodial}, $\{x,y,z\}$ is contained in a partite set $X$ of $D$.
Since $\{x,y,z\}$ forms a stable set in $C_{1,2}(D)$, $\{x,y,z\}$ is an anti-$\{1,2\}$-competing set.
Thus, by Proposition~\ref
{Prop:sub-digraph}(2), $V(D)-X$ forms a clique in $C_{1,2}(D)$ and so $X$ is the only non-$\{1,2\}$-competing partite set.
Then $X=X_1$.
Since $l \geq 6$ by our assumption, each vertex on $H$ can form an asteroidal triple with two some vertices on $H$ and so $V(H) \subset X_1$.

{\it Case 2}. $l=5$.
%{\gray
%We first assume that there is no partite set $X$ such that $|V(H) \cap X|\geq3$.}
To the contrary, suppose that there is no partite set $X$ such that $|V(H) \cap X|\geq3$.
Then, since $l =5$, there are at least three partite sets intersecting with $V(H)$.
Therefore $V(H)$ intersects with two partite sets $X_i$ and $X_j$ for some distinct $i,j\in \{2,\ldots,k\}$.
By the way, we see from the structure of $A$ that $X_i \sim X_j$.
Thus there exist a vertex in $V(H) \cap X_i$ and a vertex in $V(H) \cap X_j$ which are consecutive on $H$.
Without loss of generality,
we may assume $v_0 \in V(H)\cap  X_i $ and $v_1 \in V(H)\cap X_j $.
Then, since $v_3$ is adjacent to neither $v_0$ nor $v_1$,
the structure of $A$ shows that there is no vertex adjacent to neither $v_0$ nor $v_1$, which is impossible.
Therefore there exists a partite set $X$ such that $|V(H) \cap X|\geq 3$.
Then, since $l=5$,
there exist
two consecutive vertices on $H$ belonging to $X$.
Without loss of generality, we may assume $\{v_0,v_1\} \subset X$.
Since $l=5$, $v_0$ and $v_1$ are not true twins in $C_{1,2}(D)$ and so, by Lemma~\ref{lem:true-twin}, they are not true twins in $D$.
Then, by Lemma~\ref{lem:x_i-other-adjacent_2}, $v_3 \in X$.
Since $\{v_0,v_1,v_3\} \subseteq X$, $X$ is non-$\{1,2\}$-competing. Then $X= X_1$ or $X_2$. 

{\it Subcase 1}. $X= X_1$.
Suppose, to the contrary, that $v_4 \notin X$ and $v_2 \notin X$.
Then $v_4 \in Y$ for a partite set $Y$ distinct from $X$.
%where Y is ~ 가 거창해서 for로 바꿈.
Since $v_1$ is not adjacent to $v_4 \in Y$, we see from the structure of $A$ that $\emptyset \neq N^+(v_1)\subseteq Y$.
Therefore $v_1 \stackrel{*}{\to} v_4$ or $v_4 \stackrel{*}{\to} v_1$ by Theorem~\ref{thm:new-neighbor-stable-size2}(4).
If $v_4 \stackrel{*}{\to} v_1$, then $v_4$ is a common out-neighbor of $v_0$ and $v_3$ and so $v_0 \sim v_3$, which is impossible.
Therefore $v_1 \stackrel{*}{\to} v_4$.
By the same argument, we may show $v_0 \stackrel{*}{\to} v_2$.
If $v_2$ or $v_4$ is an out-neighbors of $v_3$,
then $v_3 \sim v_0$ or $v_3 \sim v_1$, which is impossible.
Thus $v_2$ and $v_4$ are in-neighbors of $v_3$ and so $v_2\sim v_4$, which is impossible. 
Therefore $v_2 \in X$ or $v_4 \in X$. 
Then if $v_4 \in X$, then $v_2\in X$ by Lemma~\ref{lem:x_i-other-adjacent_2} applied to $v_0$ and $v_4$. In addition, if $v_2 \in X$, then $v_4\in X$ by Lemma~\ref{lem:x_i-other-adjacent_2} applied to $v_1$ and $v_2$.
Thus \[V(H)\subseteq X.\]
%집합이 주어일 때는 belong => included로
%statement에 맞춰서 잘 고쳐볼 것.
Then, since any three vertices on $H$ do not form a triangle, each vertex in $V(D)- X$
has at most two in-neighbors in $V(H)$ and so
has at least $|V(H)|-2$ out-neighbors in $V(H)$, that is, $|N^+(w) \cap V(H) |\geq|V(H)|-2= 3$ for each vertex $w$ in $V(D)- X$.
Now, for a pair of vertices $u$ and $v$ in $V(D) - X$,
\[ \begin{aligned}
  |\left(N^+(u)\cap N^+(v)\right)  \cap V(H)| &\geq
|N^+(u)\cap V(H)| + |N^+(v) \cap V(H)|-|V(H)| \\&\geq 2 \times 3-5= 1.
 \end{aligned}
\]
Therefore any pair of vertices in $V(D) - X$ has a common out-neighbor in $V(H)$ and so $V(D) - X$ forms a clique in $C_{1,2}(D)$.
Thus $X$ is the only non-$\{1,2\}$-competing partite set of $D$. 

{\it Subcase 2}. $X=X_2$.
Then we may switch $X_1$ with $X_2$ to
still have the adjacency matrix of $C_{1,2}(D)$ in the from given in Figure~\ref{fig:structure} since Theorem~\ref{thm:structure} is applicable to any non-$\{1,2\}$-competing partite set.
Thus we may apply the above argument to conclude that $V(D)- X_2$ forms a clique in $C_{1,2}(D)$.
However,  since $X_1$ was assumed to be a non-$\{1,2\}$-competing partite set, we reach a contradiction.
Therefore $X\neq X_2$ and so $X=X_1$.
Thus
$X_1$ is the only partite set containing a hole of length at least five.
Then, since $X_1$ is the only non-$\{1,2\}$-competing partite set of $D$, 
 the part(1) is true.

For each $2\leq i \leq l-2$,
$v_0 \not \sim v_i$ and so $\{v_0,v_i\}\subseteq F_{\sigma(i)}$ for some $\sigma(i) \in \{2,\ldots,k\}$ by the structure of $A$.
Then, since $F_2,\ldots,F_k$ are mutually disjoint,
 $\{v_0,v_2,v_3,\ldots,v_{l-2}\}\subset F_{j}$ for some $j \in \{2,\ldots,k\}$.
Moreover, since $v_1 \not \sim v_3$ and $v_{l-1} \not \sim v_2$, 
$\{v_1,v_{l-1} \}\subseteq F_j$.
Thus $V(H) \subseteq F_j$.
Hence any pair of adjacent vertices in $V(H)$ must have a common out-neighbor in $X_j$ by Proposition~\ref{prop:same-partite}(3).
Then, since each vertex in $X_j$ can be a common out-neighbor of at most two vertices in $V(H)$, $|E(H)|\leq |X_j|$.
Since $|V(H)|=|E(H)|$, $|V(H)|\leq |X_j|$.
In addition,
since we have shown that $X_1$ is the only partite set containing a hole of length at least five, the part(2) is true.
\end{proof}

\begin{Thm} \label{thm:hole_four}
Let $D$ be a loose multipartite tournament and $X$ be a non-$\{1,2\}$-competing partite set.
Suppose that there is a hole  $H:=v_0v_1v_2v_3v_0$ of length four in $C_{1,2}(D)$. Then the following are true:
\begin{enumerate}[{(1)}] %통일 필요?
\item (i) $X$ is the only non-$\{1,2\}$-competing partite set in $D$ or (ii) $|X|\geq 4$ and there is exactly one non-$\{1,2\}$-competing partite set in $D$ other than $X$ and its size is at least four;
\item (i) if $v_0$ and $v_2$ (resp.\ $v_1$ and $v_3$) belong to distinct partite sets in $D$, then $v_0 \stackrel{*} \to v_2$ or $v_2 \stackrel{*} \to v_0$ (resp.\ $v_1 \stackrel{*} \to v_3$ or $v_3 \stackrel{*} \to v_1$) (ii) otherwise, $D$ has a partite set of size at least four.
\end{enumerate}
\end{Thm}
\begin{proof}
Since $H=v_0v_1v_2v_3v_0$ is a hole, 
\[v_0 \not \sim v_2 \quad \text{and} \quad v_1\not \sim v_3\]and
$v_0,v_1,v_2,v_3$ are non-sinks in $D$.
Let $X_1,\ldots,X_k$ be the partite sets of $D$ for some $k\geq 3$.

{\it Case 1.} each of $\{v_0,v_2\}$ and $\{v_1,v_3\}$ is contained in a partite set of $D$.
Then \[ \{v_0,v_2\} \subseteq X_{\alpha} \quad \text{and} \quad \{v_1,v_3\} \subseteq X_{\beta}\] for some $\alpha, \beta \in \{1,2\ldots,k\}$.
Without loss of generality, we may assume \[\alpha=1.\]
%We may assume that $X_1$ is the partite set which $v_0$ and $v_2$ belong to, that is, \[ \{v_0,v_2\} \subseteq X_1.\]
Since $v_0 \not \sim v_2$ and $v_1 \not \sim v_3$, \begin{equation}\label{eq:thm:interval2-1}
 N^+(v_0)\cap N^+(v_2)=\emptyset \quad
 \text{and} \quad N^+(v_1)\cap N^+(v_3)=\emptyset.
\end{equation} and each of $N^+(v_0) \cup N^+(v_2)$ and $N^+(v_1) \cup N^+(v_3)$ is included in a partite set by Proposition~\ref{prop:same-partite}(3). 
Then
\begin{equation}\label{eq:thm:interval2-2}
	N^+(v_0) \cup N^+(v_2) \subseteq X_{\gamma} \quad \text{and} \quad N^+(v_1) \cup N^+(v_3) \subseteq X_{\delta}
\end{equation}
for some $\gamma, \delta \in \{1,2\ldots,k\}$.
Without loss of generality, we may assume 
\[\gamma=2.\]
Then $v_0$ is an out-neighbor of the vertices in $\bigcup_{i=3}^k X_i$ and so
\statement{sta:hole_length_four_2}
{$X_i$ is a $\{1,2\}$-competing partite set for each $3\leq i \leq k$.
}

{\it Subcase 1-1}. $\beta=1$.
Then \[|X_1|\geq 4.\]
We first suppose $\delta=2$.
Then there are at least four vertices $u_0,u_1,u_2,u_3$ in $X_2$ such that $u_i$ is a common out-neighbor of $v_i$ and $v_{i+1}$ for each $0\leq i \leq 3$ (we assume $v_0=v_4$) by Proposition~\ref{prop:same-partite}(3).
Since $H$ is an induced cycle in $C_{1,2}(D)$,
$u_0,u_1,u_2,u_3$ are distinct.
Thus \[|X_2|\geq 4.\]
If $X_2$ is a non-$\{1,2\}$-competing partite set,
then $X_1$ and $X_2$ are the only non-$\{1,2\}$-competing partite sets in $D$ by \eqref{sta:hole_length_four_2}, which satisfies (ii) of the statement(1).
If $X_2$ is a $\{1,2\}$-competing partite set,
then $X_1$ is the only non-$\{1,2\}$-competing partite set by
\eqref{sta:hole_length_four_2}, which satisfies (i) of the statement(1).

Now suppose $\delta \neq 2$.
Without loss of generality, we may assume $\delta=3$.
Then $N^+(v_1)\cup N^+(v_3)\subseteq X_3$.
Thus $v_1$ is an out-neighbor of each vertex in $X_2$.
Hence $X_2$ is a $\{1,2\}$-competing partite set and so, by
\eqref{sta:hole_length_four_2}, $X_1$  
is the only non-$\{1,2\}$-competing partite set, which satisfies (i) of the statement(1).

{\it Subcase 1-2}. $\beta \neq 1$.
Then, by \eqref{sta:hole_length_four_2}, $\beta=2$, that is, 
\[ \{v_1,v_3\} \subseteq X_2.\]
By Theorem~\ref{thm:Type-AorB-distance}(4), $X_2$ is the only non-$\{1,2\}$-partite set distinct from $X_1$.
Now it suffices to show that $|X_1|\geq 4$ and $|X_2| \geq 4$ to prove the statement(1).
From \eqref{eq:thm:interval2-1} and the facts that $\{v_0,v_2\}\subseteq X_1$ and $\{v_1,v_3\} \subseteq X_2$, we may deduce that $\{v_0,v_1,v_2,v_3\}$ forms a directed cycle of length $4$.
%since가 길어서 새로운 유형 문장으로 적음
Without loss of generality, we may assume $v_0 \to v_1 \to v_2 \to v_3 \to v_0$.
Then, by \eqref{eq:thm:interval2-2}, \begin{equation}\label{eq:thm:interval2-3} N^+(v_0) \cup N^+(v_2)\subseteq X_2 \quad \text{and} \quad N^+(v_1) \cup N^+(v_3)\subseteq X_1\end{equation}
Since $H$ is a cycle,
$v_i \sim v_{i+1}$ for each $0\leq i \leq 3$ (we assume $v_0=v_4$).
Thus
$v_i \not \stackrel{*}{\to} v_{i+1}$ for each $0\leq i \leq 3$ by Corollary~\ref{Cor:cor-adajcent} and so $v_{i}$ has an out-neighbor $w_i$ distinct from $v_{i+1}$ for each $0\leq i \leq 3$.
Then $\{w_1,w_3\}\subseteq X_1$ and $\{w_2,w_4\}\subseteq X_2$ by \eqref{eq:thm:interval2-3} and so $\{v_0,v_2,w_1,w_3\}\subseteq X_1$ and
$\{v_1,v_3,w_2,w_4\}\subseteq X_2$.
In addition, by \eqref{eq:thm:interval2-1},
$w_1 \neq w_3$ and  $w_2 \neq w_4$ and so $|X_1|\geq 4$ and $|X_2|\geq 4$.
Hence the statement(1) is true.

Further, we have shown that $D$ has a non-$\{1,2\}$-competing partite set of size at least four in each subcase, which implies the ``otherwise" part of the statement(2) is true.

%{\blue $\gamma, \alpha$등 반영할 것 1/5}
{\it Case 2.} either $v_0$ and $v_2$ or $v_1$ and $v_3$ belong to distinct partite sets.
We first claim 
\statement{sta:hole_length_four_new1}
{
if $v_0$ and $v_2$ (resp.\ $v_1$ and $v_3$) belong to distinct partite sets in $D$,
then $v_0 \stackrel{*} \to v_2$ or $v_2 \stackrel{*} \to v_0$ (resp.\ $v_1 \stackrel{*} \to v_3$ or $v_3 \stackrel{*} \to v_1$).
}
To the contrary, we suppose $v_0$ and $v_2$ belong to distinct partite sets, $v_0 \not \stackrel{*} \to v_2$, and $v_2 \not \stackrel{*} \to v_0$.
Without loss of generality, we may assume $v_0\in X_1, v_2 \in X_2$, and $v_0 \to v_2$.
Then, by Proposition~\ref{Prop:different-partite-adjacent}, $N^+(v_0)\cap N^+(v_2)=\emptyset$ and there exists a partite set $X$, such that $N^+(v_2)\subseteq X$ and $N^+(v_0)\subseteq X \cup\{v_2\}$.
Since $v_2$ is a non-sink, $X\neq X_2$.
In addition, since $v_0$ has an out-neighbor distinct from $v_2$, $X\neq X_1$.
%Thus we may assume $X=X_3$.
Then we can easily check that every vertex in $V(D)\setminus \{v_0,v_2\}$ has $v_0$ or $v_2$ as an out-neighbor.
Thus, by the assumption that $v_0 \to v_2$, $v_1$ and $v_3$ $\{1,2\}$-compete, which is impossible.
Therefore $v_0 \stackrel{*} \to v_2$ or $v_2 \stackrel{*} \to v_0$.
Then, by symmetry of $H$, if $v_1$ and $v_3$ belong to distinct partite sets in $D$, then $v_1 \stackrel{*} \to v_3$ or $v_3 \stackrel{*} \to v_1$.
Hence we conclude that \eqref{sta:hole_length_four_new1} holds.

Now, by the case assumption, we suppose that
$v_0$ and $v_2$ belong to distinct partite sets.
Without loss of generality, we may assume $v_0 \stackrel{*} \to v_2$ by \eqref{sta:hole_length_four_new1} and $v_0 \in X_1$ and $v_2 \in X_2$.
If $\{v_1,v_3\} \subseteq \bigcup_{i=2}^k X_i$, 
then $v_0$ is an out-neighbor of $v_1$ and $v_3$ and so $v_1$ and $v_3$ $\{1,2\}$-compete, which is impossible.
Thus $\{v_1,v_3\} \not \subseteq \bigcup_{i=2}^kX_i$.
Then, without loss of generality,
we may assume $v_1 \in X_1$.

Suppose $v_3 \notin X_1$.
Then, by \eqref{sta:hole_length_four_new1}, $v_1 \stackrel{*}\to v_3$ since $v_0$ is an out-neighbor of $v_3$.
Take two vertices $x,y$ in $V(D)\setminus X_1$.
If $\{x,y\}=\{v_2,v_3\}$, then
$x$ and $y$ are adjacent by $H$.
If $\{x,y\}\neq \{v_2,v_3\}$, then
$v_0$ or $v_1$ is a common out-neighbor of $x$ and $y$.
%Thus
%$V(D)-X_1$ forms a clique.
Thus any pair of vertices in
$V(D)\setminus X_1$ compete.
Hence $X_1$ is the only non-$\{1,2\}$-competing partite set in $D$.

Now suppose $v_3 \in X_1$.
Then, by Proposition~\ref{prop:same-partite}, there exists a partite set $X$ such that $N^+(v_1) \cup N^+(v_3) \subseteq X$.
If $X=X_2$, then $v_1 \to v_2$ and $v_3 \to v_2$ since $v_0 \stackrel{*}\to v_2$, which contradicts $v_1 \not \sim v_3$.
Thus $X \neq X_2$ and so $v_1 \not \to v_2$ and $v_3 \not \to v_2$.
We note $N^+(v_1)\cap N^+(v_3)=\emptyset$.
Take a vertex $z$ in $V(D)\setminus X_1$.
Then the following are true:
\begin{itemize}
\item if $z \notin N^+(v_0)\cup N^+(v_1) \cup N^+(v_3)$, then $\{v_0,v_1,v_3\} \subseteq N^+(z)$;
\item if $z \in N^+(v_0)=\{v_2\}$, then $\{v_1,v_3\}\subseteq N^+(z)$;
\item if $z\in N^+(v_1)$, then $\{v_0,v_3\}\subseteq N^+(z)$;
\item if $z\in N^+(v_3)$, then $\{v_0,v_1\} \subseteq N^+(z)$.
\end{itemize}
Thus any pair of vertices in
$V(D)\setminus X_1$ compete and so $X_1$ is the only non-$\{1,2\}$-competing partite set in $D$.
\end{proof}

\begin{Cor}
	Let $D$ be a loose multipartite tournament.
	Suppose that $D$ has no vertex of outdegree one and each partite set have set has sizes at most three.
	Then $C_{1,2}(D)$ is chordal.
\end{Cor} 
\begin{proof}
To the contrary, suppose $C_{1,2}(D)$ is not chordal.
Then $C_{1,2}(D)$ has a hole $H$.
By Theorem~\ref{thm:hole-charc}(2),
$H$ cannot have length at least five.
Thus $H$ has length four.
Then, by Theorem~\ref{thm:hole_four}(2),
$D$ has a vertex of outdegree one
or $D$ has a partite set of size at least four, which contradicts the hypothesis.
Therefore $C_{1,2}(D)$ has no hole and so is chordal.
\end{proof}
\begin{Thm}\label{thm:partite_size_interval}
	Let $D$ be a loose multipartite tournament.
	If each partite set has size at most two, then $C_{1,2}(D)$ is an interval graph.
\end{Thm}
\begin{proof}
	Suppose, to the contrary, that every partite set has size at most two and $C_{1,2}(D)$ is not interval.
	Then $C_{1,2}(D)$ has a hole or an asteroidal triple (recall that a graph is an interval graph if and only if it is chordal and contains no asteroidal triple by \cite{lekkeikerker1962representation}).	
	If $C_{1,2}(D)$ has an asteroidal triple, then there is a partite set of size at least three by Lemma~\ref{lem-asterodial}, which is impossible.
	Thus $C_{1,2}(D)$ has no asteroidal triple and so it has a hole $H$.
	Then, by Theorem~\ref{thm:hole-charc}(2), $H$ has length four.	
	Let $H=v_0v_1v_2v_3v_0$.
	Then none of $\{v_0,v_2\}$ and $\{v_1,v_3\}$ belongs to one partite set by (ii) of Theorem~\ref{thm:hole_four}(2).
	Thus $v_0$ and $v_2$ belong to distinct partite sets, say $X_1$ and $X_2$, respectively. 
	 In addition, by (i) of Theorem~\ref{thm:hole_four}(2),\[
(a)\text{ }v_0 \stackrel{*} \to v_2  \text{ or } v_2 \stackrel{*} \to v_0 \quad \text{and} \quad (b)\text{ } v_1 \stackrel{*} \to v_3 \text{ or } v_3 \stackrel{*} \to v_1. \]
By the symmetry of $H$,
we may assume
		Without loss of generality, we may assume
	\begin{equation}\label{eq:thm:size_interval}	
v_0 \stackrel{*} \to v_2 \quad \text{and} \quad   v_1 \stackrel{*} \to v_3.\end{equation}
If $v_1 \notin X_1$, then $v_1 \to v_0$, contradicting \eqref{eq:thm:size_interval}.
Thus $v_1 \in X_1$ and so $X_1=\{v_0,v_1\}$.
	Then $v_3 \notin X_1$ and so, by \eqref{eq:thm:size_interval}, $v_3 \to v_0$.
By \eqref{eq:thm:size_interval}, $v_0$ and $v_1$ have no common out-neighbor. 
If $v_3 \in X_2$, then, by Proposition~\ref{prop:same-partite}(3), $v_0 \not \sim v_1$, which is impossible.
Thus $v_3 \not \in X_2$.
Let $X_3$ be the partite set containing $v_3$.
Then $v_1$ is a common out-neighbor of the vertices in $V(D)\setminus (X_1\cup X_3)$ by \eqref{eq:thm:size_interval}.
Therefore each partite set except $X_1$ and $X_3$ is $\{1,2\}$-competing.
By the way, $X_1$ and $X_3$ are also $\{1,2\}$-competing.
To see why, we note that $v_0$ is a common out-neighbor of $X_3$ by \eqref{eq:thm:size_interval} and so $X_3$ is $\{1,2\}$-competing.
Moreover, since $X_1=\{v_0,v_1\}$ and $v_0 \sim v_1$, 
$X_1$ is $\{1,2\}$-competing.
Thus we have reached a contradiction to the hypothesis that $D$ is a loose multipartite tournament.
\end{proof}

A graph is called {\it $C_4$-free} if it contains no hole of length four.
\begin{Thm} \label{thm:C_4_equivalent}
	Let $D$ be a multipartite tournament such that $D$ has exactly two non-$\{1,2\}$-competing partite sets. Then the following are equivalent:
\begin{itemize}
\item[(i)] $C_{1,2}(D)$ is an interval graph;
\item[(ii)] $C_{1,2}(D)$ is chordal;
\item[(iii)] $C_{1,2}(D)$ is $C_4$-free.
\end{itemize}
In particular, if, except for one partite set of $D$, all have sizes of three or less, then $C_{1,2}(D)$ is an interval graph.
\end{Thm}
\begin{proof}
Recall that a graph is an interval graph if and only if it is chordal and contains no asteroidal triple by \cite{lekkeikerker1962representation} and so (i) implies (ii).
	In addition, since chordal graph has no induced cycle of length four, (ii) implies (iii).
	Now suppose that (iii) is true.
	%holds는 수식에서/ is true로
	Since $D$ has more than one $\{1,2\}$-competing partite set,
	$C_{1,2}(D)$ has no asteroidal triple by Lemma~\ref{lem-asterodial} and, by Theorem~\ref{thm:hole-charc}(1), $C_{1,2}(D)$ has no hole of length at least five.	
	Thus $C_{1,2}(D)$ is an interval graph.	
	
	To show the ``in particular" part, suppose that except for one partite set of $D$, all have sizes of three or less,	$C_{1,2}(D)$ is $C_4$-free by Theorem~\ref{thm:hole_four}(1).
	Hence $C_{1,2}(D)$ is an interval graph by (i) $\Leftrightarrow$ (iii).\end{proof}

\section{Acknowledgement}

This work was supported by Science Research Center Program through the National Research Foundation of Korea (NRF) grant funded by the Korean Government (MSIT) (NRF-2022R1A2C1009648 and 2016R1A5A1008055).

\bibliographystyle{plain}
%\bibliography{12-step}

\end{document}